\documentclass[11pt]{article}

\usepackage[margin=30mm]{geometry}
\usepackage{amsmath,amssymb,amsthm,mathtools}
\usepackage{enumitem}
\usepackage[colorlinks=true,citecolor=blue,linkcolor=blue,urlcolor=blue]{hyperref}

\newtheorem{theorem}{Theorem}[section]
\newtheorem{lemma}[theorem]{Lemma}

\newtheorem{observation}[theorem]{Observation}

\newcommand{\VD}{\mathrm{VD}}
\newcommand{\ex}{\operatorname{ex}}
\newcommand{\extr}{\operatorname{ex}_{\mathrm{tr}}}
\newcommand{\pow}{\mathcal P}
\newcommand{\OO}{O}
\newcommand{\Om}{\Omega}

\title{Venn diagrams as forbidden hypergraph traces}
\author{Adam Džavoronok \footnote{Department of Applied Mathematics, Charles University, Faculty of Mathematics and Physics, Malostransk\'e n\'am.~25, 118 00~ Praha 1, Czech Republic, \texttt{adam.dzavoronok@mff.cuni.cz},\texttt{
timreizin@gmail.com}}
	\and
 Tymofii Reizin \footnotemark[1]
\and     
Jakub Šošovička \footnote{Department of Algebra, Charles University, Faculty of Mathematics and Physics
Sokolovsk\'a 83, 186 75 Praha 8, Czech Republic, \texttt{jakubsosovicka73@gmail.com}}
}

\begin{document}
\maketitle

\begin{abstract}
We study the maximum size of a set system that contains no $k$-Venn diagram, denoted by $\VD_k$, as a trace. For every fixed $k\ge 3$, we prove $\extr(n,\VD_k)=\OO_k(n^{2^k-2k+1})$, improving the direct Sauer-Shelah bound $\OO_k(n^{2^k-1})$.  In particular, for $k=4$ the exponent decreases from $15$ to $9$. The proof starts from the theorem of Keevash, Leader, Long and Wagner for $\VD_3$ and uses induction on $k$ in which two new Venn regions are forced for free at each added edge. We also record lower-bound constructions for Venn diagrams in fixed uniformity, explicit bounds for the $4$-uniform $3$-Venn problem, and a fixed uniformity trace result for the loose triangle.
\end{abstract}

\section{Introduction}

A $k$-Venn diagram is the incidence pattern formed by $k$ sets when all $2^k$ possible regions are non-empty. We are interested in how large a family of subsets of an $n$-element set can be if it contains no such configuration, even after passing to a trace. We use the following trace language. If $H$ is a hypergraph and $S\subseteq V(H)$, the trace of $H$ on $S$ is $H|_S=\{E\cap S:E\in E(H)\}$. We say that $H$ contains a hypergraph $F$ as a trace if $H|_S$ contains a subhypergraph isomorphic to $F$ for some $S\subseteq V(H)$. Let $\extr(n,F)$ denote the maximum number of edges in an $n$-vertex hypergraph with no trace copy of $F$. We lastly remark that, if we do not explicitly state the uniformity of the hypergraph $H$, we treat it as a set system $E(H)\subseteq\mathcal{P}(V(H))$.

The $k$-Venn diagram, which we denote as $\VD_k$, will be encoded as a $k$-edge hypergraph. Its vertex set is $2^{[k]}$, and its $i$-th edge is $E_i=\{A\subseteq [k]:i\in A\}$. Thus, a vertex $A\subseteq [k]$ represents the region contained in exactly the Venn sets whose indices lie in $A$. We write $\VD_k^\circ$ for the hypergraph obtained from $\VD_k$ by deleting the outside vertex $\emptyset$.

This is also a matrix forbidden-configuration problem. If $\mathcal F\subseteq \mathcal P([n])$, we view $\mathcal F$ as the $0$-$1$ incidence matrix with rows indexed by elements of $[n]$ and columns indexed by sets in $\mathcal F$. Then a $k$-Venn diagram in $\mathcal F$ is exactly a copy, up to row and column permutations, of the $2^k\times k$ matrix whose rows are all binary vectors of length $k$. Equivalently, we say that the dual VC~-~dimension of $\mathcal F$ is at least $k$. This viewpoint was introduced by Keevash, Leader, Long and Wagner in their work on the case $k=3$ \cite{KLLW}.

The Sauer–Perles–Shelah–Vapnik–Chervonenkis theorem \cite{Sauer,PerlesShelah,VC} yields the immediate bound\(\extr(n,\VD_k)=\OO_k(n^{2^k-1}).\)
Indeed, if a family shatters a set of size $2^k$, then $k$ suitably chosen members realize all $2^k$ incidence patterns and hence form a trace copy of \(\VD_k\). Thus every \(\VD_k\)-free family has VC-dimension at most $2^k-1$. The opposite construction is also simple: take all subsets of $[n]$ of size smaller than $2^{k-1}$. No member of this family can play the role of an edge of a $k$-Venn diagram, since every edge of $\VD_k$ has size $2^{k-1}$. Hence $\extr(n,\VD_k)=\Om_k(n^{2^{k-1}-1})$.

The gap between these two bounds is already non-trivial for $k=3$, whereas the case $k=2$ is a folklore problem of forbidding two crossing sets and has a linear extremal number. For $k=3$, Gupta, Lee and Li proved an upper bound of order $n^{3.75}$ as part of their work on minimum $k$-cuts \cite{GuptaLeeLi}. Keevash, Leader, Long and Wagner then proved the sharp estimate $\extr(n,\VD_3)=\Theta(n^3)$ \cite{KLLW}.

The conjectural general picture is supplied by the Anstee–Sali framework for forbidden configurations. For a fixed forbidden matrix $M$, the Anstee–Sali conjecture predicts that the extremal number is governed by a product construction using singleton, cosingleton, and chain factors; see the survey \cite{AnsteeSaliSurvey}. The relevant parameter is hard to compute in general \cite{Raggi}, but for the Venn matrices above the conjecture predicts precisely the lower-bound exponent $2^{k-1}-1$. Thus the expected answer is $\extr(n,\VD_k)=\Theta_k(n^{2^{k-1}-1})$. Our main result gives a polynomial improvement over the direct Sauer–Shelah bound for every fixed $k$.

\begin{theorem}[General trace bound]\label{thm:main-trace}
For every fixed integer $k\ge 3$, we have $\extr(n,\VD_k)=\OO_k(n^{2^k-2k+1})$.
\end{theorem}

 The improvement is already substantial in the first open case: for $k=4$, the direct bound is $O(n^{15})$, whereas our result gives $O(n^9)$ and the conjectured exponent is $7$. The proof starts from the theorem of Keevash, Leader, Long and Wagner for $\VD_3$~\cite{KLLW} and then proceeds by induction. The key step is to show that, once a $\VD_k$ trace can be forced, the same exponent already forces $(k+1)$ sets which realize $2^k+1$ non-outside regions of $\VD_{k+1}$ . Compared with the full non-outside $k$-Venn diagram, this gives two additional regions for free. The remaining regions are then added one at a time by a standard trace-extension argument. 
We also record two fixed-uniformity results. Let $\ex_r(n,F)$ be the maximum number of edges in an $r$-uniform $n$-vertex hypergraph containing no copy of $F$ as a subhypergraph.

\begin{theorem}[Uniform lower bound]\label{thm:uniform-lower}
For every fixed $k\ge 2$, $\ex_{2^{k-1}}(n,\VD_k^\circ)=\Om_k(n^{2^{k-1}-1})$. Moreover, for every fixed $r\ge 2^{k-1}$ there are $r$-uniform hypergraphs with $\Om_{k,r}(n^{2^{k-1}-1})$ edges which avoid $\VD_k$ as a trace.
\end{theorem}

For the first non-trivial uniform case we obtain the following explicit estimates.

\begin{theorem}[The $4$-uniform $3$-Venn problem]\label{thm:four-uniform}
For every \(n\ge 8\),
\[
\sum_{i=1}^{\lfloor n/2\rfloor}\binom{n-2i}{2}\leq \ex_4(n,\VD_3^\circ)\le \frac{n}{4}\binom{n-2}{2}.
\]
\end{theorem}

We also study the loose triangle $C_3^3$, a natural subconfiguration of $\mathrm{VD}_3^\circ$ obtained by removing the vertex in all three edges. Its trace problem provides a simple example in which the unrestricted and fixed-uniformity extremal numbers have different orders.

\begin{theorem}[Fixed-uniformity trace problem for the loose triangle]\label{thm:loose-triangle-trace}
For every fixed integer $r\ge 3$, the maximum number of edges in an $r$-uniform $n$-vertex hypergraph with no trace copy of $C_3^3$ is $\Theta_r(n^2)$.
\end{theorem}

The paper is organised as follows. Section~\ref{sec:trace} proves the main trace bound. Section~\ref{sec:uniform} gives the fixed-uniformity Venn constructions and proves Theorem~\ref{thm:four-uniform}. Section~\ref{sec:loose} proves Theorem~\ref{thm:loose-triangle-trace}.

\section{The trace bound}\label{sec:trace}

Throughout this section, containment of a fixed hypergraph with labeled edges as a trace means that the selected members of the ambient family realize the prescribed incidences with those labeled edges.

For sets $A_1,\ldots,A_t\subseteq [n]$ and $I\subseteq [t]$, put
\[
R_I(A_1,\ldots,A_t)=\left(\bigcap_{i\in I}A_i\right)\cap\left(\bigcap_{i\in [t]\setminus I}([n]\setminus A_i)\right).
\]
We say that $A_1,\ldots,A_t$ realize $m$ non-outside regions if at least $m$ of the regions $R_I(A_1,\ldots,A_t)$ with $I\ne\emptyset$ are non-empty. In this section and later on, we work with the notion of a loose cycle, which we now define. For \(r\ge 2\) and \(\ell\ge 3\), the \(r\)-uniform loose cycle \(C_\ell^r\) consists of edges
\(T_1,\ldots,T_\ell\) and distinct vertices \(v_1,\ldots,v_\ell\) such that
\(
T_i\cap T_{i+1}=\{v_i\}
\)
with indices modulo \(\ell\), and all other intersections are empty. In particular, \(C_3^3\) has six vertices and an empty triple intersection. Moreover, we use the following extremal result as a black box.
\begin{theorem}[Füredi,Jiang \cite{FurediJiang}]\label{thm:loose cycle r,l}
    For fixed $r\ge3$ and $\ell\ge3$ $$\ex_r(n,C_\ell^r)=\OO_{r,\ell}(n^{r-1}).$$
\end{theorem}

We briefly sketch the inductive structure of the proof of Theorem~\ref{thm:main-trace}.
The base case is the theorem of Keevash, Leader, Long and Wagner ~\cite{KLLW} for
\(\VD_3\). The induction uses two ingredients. First, we use the standard trace-extension argument (see \cite{AnsteeSaliSurvey,KLLW}) proved as Lemma~\ref{lem:trace-extension}, which states that once a fixed family of labeled trace configurations is forced at exponent \(q\), adding one prescribed trace vertex costs one further power of \(n\). Second, Lemma~\ref{lem:region-splitting} shows that if \(\VD_k\) is forced at exponent
\(q_k=2^k-2k+1\), then the same exponent already forces \(k+1\) sets realising
at least \(2^k+1\) non-outside regions. The key point is to find a trace copy
of \(\VD_k\) using sets of size different from \(2^{k-1}\), so that some
region has  at least two vertices in it.

This allows for an induction on the ground
set after separating sets that split an old region from those that are
constant on all old non-outside regions. Similar argument was also used by Keevash, Leader, Long and Wagner~\cite{KLLW} in their structural analysis for $\VD_3$. The only exceptional case is when many sets have the critical size
\(2^{k-1}\); this is handled separately in Lemma~\ref{lem:critical-size}.
After that, Lemma~\ref{lem:region-splitting} supplies \(2^{k-1}+1\) non-outside regions in the
step from \((k-1)\) to \(k\). The trace-extension Lemma~\ref{lem:trace-extension} adds the remaining
regions one by one, and the outside region is forced at the end.

\begin{lemma}\label{lem:trace-extension}
Let $\mathcal G$ and $\mathcal G^+$ be finite families of hypergraphs whose edges are indexed by the same finite label set. Suppose that for every $G\in\mathcal G$, some member of $\mathcal G^+$ is obtained from $G$ by adding one new vertex and prescribing, for each edge label, whether the new vertex belongs to that edge. If every family avoiding all members of $\mathcal G$ as traces has size $\OO(n^q)$, then every family avoiding all members of $\mathcal G^+$ as traces has size $\OO(n^{q+1})$.
\end{lemma}

\begin{proof}
Let $f(n)$ and $g(n)$ be the corresponding extremal functions for $\mathcal G^+$ and $\mathcal G$, respectively. Fix a vertex $x$ of the ground set and project each set $E$ to $E\setminus{x}$. Let $\mathcal F_1$ be the projected sets that have exactly one lift in the original family, and let $\mathcal F_2$ be the projected sets that have two lifts. The family $\mathcal F_1$ avoids every member of $\mathcal G^+$ as a trace, since any such trace would already occur away from $x$. The family $\mathcal F_2$ avoids every member of $\mathcal G$ as a trace. Indeed, if it contained a trace copy of some $G\in\mathcal G$, choose $G^+\in\mathcal G^+$ obtained from $G$ by adding a vertex with a prescribed incidence pattern. Since both lifts of every selected edge are available, we may choose the lift containing $x$ precisely for those labeled edges that contain the new vertex of $G^+$. The selected lifts, together with $x$, then give a trace copy of $G^+$, contradiction. Thus
\(
f(n)\le f(n-1)+2g(n-1).
\)
Iterating this recurrence and using $g(n)=\OO(n^q)$ gives
$f(n)=\OO(n^{q+1})$.
\end{proof}
\begin{lemma}\label{lem:critical-size}
Let $k\ge3$ and put $q_k=2^k-2k+1$. There is a constant $C=C(k)$ such that every family $\mathcal F\subseteq \binom{[n]}{2^{k-1}}$ of size at least $Cn^{q_k}$ contains $k+1$ members realising at least $2^k+1$ non-outside regions.
\end{lemma}

\begin{proof}
For $k\ge4$, we have $q_k\ge 2^{k-1}$, so the whole layer $\binom{[n]}{2^{k-1}}$ has size $\OO_k(n^{q_k})$. Taking $C$ larger than the implicit constant makes the assertion vacuous.

It remains to consider $k=3$. Let $\mathcal F\subseteq\binom{[n]}4$ have size at least $Cn^3$. For $C$ large enough, some vertex $x$ belongs to at least $c n^2$ members of $\mathcal F$, where $c$ is larger than the constant in the bound $\ex_3(n,C_4^3)=\OO(n^2)$ from Theorem~\ref{thm:loose cycle r,l}. Hence the link $L_x=\{E\setminus\{x\}:x\in E\in\mathcal F\}$ contains a copy of the loose $4$-cycle $C_4^3$, say with edges $T_1,T_2,T_3,T_4$ in cyclic order. Put $E_i=T_i\cup\{x\}$ for $i=1,2,3,4$. The point $x$ gives the region with pattern $\{1,2,3,4\}$. Each consecutive intersection $T_i\cap T_{i+1}$ gives a region with pattern $\{i,i+1\}$, with indices taken modulo $4$. Finally, each $T_i$ has one vertex outside $T_{i-1}\cup T_{i+1}$, and this gives a singleton region with pattern $\{i\}$. Thus $E_1,E_2,E_3,E_4$ realize at least $1+4+4=9$ non-outside regions, as required.
\end{proof}

\begin{lemma}\label{lem:region-splitting}
Let $k\ge3$ and put $q_k=2^k-2k+1$. Suppose that $\extr(n,\VD_k)=\OO_k(n^{q_k})$. Then there is a constant $C=C(k)$ such that every family $\mathcal F\subseteq\pow([n])$ of size at least $Cn^{q_k}$ contains $k+1$ members realising at least $2^k+1$ non-outside regions.
\end{lemma}

\begin{proof}
Let $C_0$ be such that every family of at least $C_0n^{q_k}$ sets contains a trace copy of $\VD_k$, and let $C_1$ be the constant from Lemma~\ref{lem:critical-size}. We prove the lemma by induction on $n$, with the final constant $C$ chosen large enough in terms of $k,C_0,C_1$ and the finitely many initial values of $n$. Assume that $\mathcal F\subseteq\pow([n])$ contains no $k+1$ members realising $2^k+1$ non-outside regions. We show that $|\mathcal F|<Cn^{q_k}$. If fewer than $C_0n^{q_k}$ members of $\mathcal F$ have size different from $2^{k-1}$, then either $|\mathcal F|<(C_0+C_1)n^{q_k}$ or the critical layer $\binom{[n]}{2^{k-1}}$ contains at least $C_1n^{q_k}$ members, contradicting Lemma~\ref{lem:critical-size}.

We may therefore choose $k$ members $A_1,
\ldots,A_k$ of size different from $2^{k-1}$ which form a trace copy of $\VD_k$. Write $R_I=R_I(A_1,
\ldots,A_k)$. Each $R_I$ is non-empty. Since every $A_j$ contains at least one vertex from each of the $2^{k-1}$ regions $R_I$ with $j\in I$, but $|A_j|\ne 2^{k-1}$, we have $|A_j|>2^{k-1}$. Hence at least one non-outside region has size at least two.

A set $B\subseteq[n]$ splits a region $R_I$ if both $B\cap R_I$ and $R_I\setminus B$ are non-empty. Let $s(B)$ be the number of non-outside regions split by $B$, and let $\eta(B)=1$ if $B\cap R_\emptyset\ne\emptyset$, and $\eta(B)=0$ otherwise. The sets $A_1,
\ldots,A_k,B$ realize at least $(2^k-1)+s(B)+\eta(B)$ non-outside regions. By the assumption on $\mathcal F$, we must have $s(B)+\eta(B)\le1$ for every $B\in\mathcal F$. In particular, any member which splits a non-outside region is disjoint from $R_\emptyset$ and splits at most one non-outside region.

Let $\mathcal F_1$ be the subfamily of all members that split at least one non-outside region, and put $r=|R_\emptyset|$. Then $r\ge1$ and $\mathcal F_1$ is contained in $[n]\setminus R_\emptyset$. By the induction hypothesis applied on this smaller ground set, $|\mathcal F_1|\le C(n-r)^{q_k}$.

Let $\mathcal F_2=\mathcal F\setminus\mathcal F_1$. No member of $\mathcal F_2$ splits a non-outside region. Compress every non-outside region $R_I$, $I\ne\emptyset$, to one vertex $x_I$, and keep the vertices of $R_\emptyset$. For $B\in\mathcal F_2$, define $B^*=(B\cap R_\emptyset)\cup\{x_I:R_I\subseteq B\}$. The map $B\mapsto B^*$ is injective, because the members of $\mathcal F_2$ are constant on every non-outside region. Moreover, every trace configuration in the compressed family lifts to the same trace configuration in $\mathcal F_2$ by choosing one representative from each compressed region used by the trace. Since at least one non-outside region has size at least two, the compressed ground set has size $r+2^k-1\le n-1$. Therefore, the induction hypothesis gives $|\mathcal F_2|\le C(r+2^k-1)^{q_k}$.

Combining the two bounds gives $|\mathcal F|\le C((n-r)^{q_k}+(r+2^k-1)^{q_k})$. Here $1\leq r\leq n-2^k$. The function
$$h(r)=(n-r)^{q_k}+(r+2^k-1)^{q_k}$$
is convex on this interval, so its maximum is attained at an endpoint. Therefore,
$$h(r)\leq (n-1)^{q_k}+2^{kq_k}
=n^{q_k}-\Omega_k\bigl(n^{q_k-1}\bigr)+O_k(1)
<n^{q_k}$$
for all sufficiently large $n$. Increasing $C$ to cover the remaining values of $n$ completes the induction.
\end{proof}

\begin{proof}[Proof of Theorem~\ref{thm:main-trace}]
The case $k=3$ is the theorem of Keevash, Leader, Long and Wagner~\cite{KLLW}. Now let $k\ge4$ and assume the result for $(k-1)$, with exponent $q_{k-1}=2^{k-1}-2(k-1)+1$. 

For $1\le s\le 2^k-1$, let $\mathcal G_s$ be the finite family of all labeled $k$-edge hypergraphs obtained by restricting $\VD_k^\circ$ to an $s$-element subset of its vertex set, while retaining the edge labels. Lemma~\ref{lem:region-splitting} applied with $(k-1)$ shows that the maximum size of a family containing no member of $\mathcal G_{2^{k-1}+1}$ as a trace is
\(\OO_k(n^{q_{k-1}})\). Applying Lemma~\ref{lem:trace-extension} successively $2^{k-1}-2$ times therefore shows that the maximum size of a family containing no trace copy of $\VD_k^\circ$ is
\[
\OO_k\left(n^{q_{k-1}+2^{k-1}-2}\right)
=
\OO_k\left(n^{2^k-2k+1}\right)
=
\OO_k(n^{q_k}).
\]

It remains to force the outside region. Let $D=D(k)$ be such that every family of more than $Dn^{q_k}$ sets contains a trace copy of $\VD_k^\circ$, and let $\mathcal F$ have more than $2Dn^{q_k}$ members. Fix a vertex $x$. Then at least half of the members of $\mathcal F$ avoid $x$ or at least half contain $x$. In the first case, applying the $\VD_k^\circ$ statement to the subfamily avoiding $x$ gives all non-outside regions, and $x$ supplies the outside region. In the second case, complement the sets containing $x$, apply the $\VD_k^\circ$ statement, and then complement back. The obtained patterns are all patterns except $[k]$, while $x$ supplies the all-one pattern. Thus we obtain a trace copy of $\VD_k$.
\end{proof}

\section{Venn diagrams in fixed uniformity}\label{sec:uniform}

We recall a classical theorem of Erd\H{o}s \cite{ErdosTuran}: for every fixed $r$-partite $r$-uniform hypergraph $F$, there exists $\varepsilon=\varepsilon(F)>0$ such that
\(
\ex_r(n,F)=O_F(n^{r-\varepsilon}).
\)
Equivalently, a fixed $r$-uniform hypergraph has Tur\'an density zero if and only if it is $r$-partite. For $\VD_k^\circ$, pair each nonempty proper subset $I\subsetneq[k]$ with its complement $[k]\setminus I$, and assign one color to each complementary pair. Assign one additional color to the vertex $[k]$. This gives an $r$-partite representation with $r=2^{k-1}$, since every edge of $\VD_k^\circ$ contains exactly one vertex of each color. Hence, for some $\varepsilon=\varepsilon(k)>0$,
\(
\ex_{2^{k-1}}(n,\VD_k^\circ)
= O_k\bigl(n^{2^{k-1}-\varepsilon}\bigr).
\) The following construction shows that the lower-bound exponent from the trace problem also appears in the fixed-uniformity problem.

\begin{proof}[Proof of Theorem~\ref{thm:uniform-lower}]
Partition the ground set as $X\cup Y$, where $|X|,|Y|=\Theta(n)$, and choose a matching $M$ of size $\Theta(n$) in $Y$.  Define a $2^{k-1}$-uniform hypergraph by taking all sets $A\cup P$, where $A\in\binom{X}{2^{k-1}-2}$ and $P\in M$. This gives $\Om_k(n^{2^{k-1}-1})$ edges.

Suppose that $k$ selected edges formed a copy of $\VD_k^\circ$. Choose one of them and let $P=\{p,p'\}$ be its matching pair. Every selected edge contains either both $p,p'$ or neither of them. Thus $p$ and $p'$ have the same incidence pattern with respect to the selected $k$ edges. This is impossible in $\VD_k^\circ$, where all vertices have distinct incidence patterns.

For the trace statement in uniformity $r\ge 2^{k-1}$, replace the matching pairs by disjoint blocks of size $r-2^{k-1}+2$, and take all unions of such a block with a $(2^{k-1}-2)$-subset of $X$. The number of edges is again $\Om_{k,r}(n^{2^{k-1}-1})$. If $k$ selected edges realized a trace copy of $\VD_k$, then the trace of each selected edge would have size $2^{k-1}$. But each selected edge has only $2^{k-1}-2$ vertices in $X$, so at least two vertices from its block would have to appear in the trace. Those two vertices have identical incidence pattern with respect to all selected edges, contradicting the fact that the vertices of $\VD_k$ have distinct incidence patterns.
\end{proof}
For the proof of Theorem~\ref{thm:four-uniform}, let us state the following extremal result, first proved by Frankl and F\"uredi for $n\ge75$ and later extended by Cs\'ak\'any and Kahn.
\begin{theorem}[\cite{FranklFuredi,CsakanyKahn}]\label{thm:exact_loose cycle}
For \(n\ge 7\),
\(
\ex_3(n,C_3^3)=\binom{n-1}{2}.
\)
\end{theorem}
\begin{proof}[Proof of Theorem~\ref{thm:four-uniform}]
For the upper bound, let $\mathcal F\subseteq\binom{[n]}4$ and assume $|\mathcal F|>\frac{n}{4}\binom{n-2}{2}$. Some vertex $x$ is contained in more than $\binom{n-2}{2}$ members of $\mathcal F$. The link of $x$ is a $3$-uniform hypergraph on $n-1$ vertices with more than $\binom{n-2}{2}$ edges. By Theorem~\ref{thm:exact_loose cycle} for the loose triangle, this link contains a copy of $C_3^3$. Adding $x$ to the three link edges gives a copy of $\VD_3^\circ$.

For the lower bound, expose the vertices in disjoint pairs. First choose a pair $P_1$ and add all $4$-sets $P_1\cup Q$, where $Q$ ranges over all pairs among the remaining $(n-2)$ vertices. Then choose a pair $P_2$ from the unused vertices and add all $P_2\cup Q$, where $Q$ avoids $P_1\cup P_2$, and continue. This gives exactly $\sum_{i=1}^{\lfloor n/2\rfloor}\binom{n-2i}{2}$ edges.

Suppose that three selected edges formed a copy of $\VD_3^\circ$. Choose among them an edge whose distinguished pair $P_i=\{p,p'\}$ was exposed earliest. Any selected edge exposed later is disjoint from $P_i$, while any selected edge with distinguished pair $P_i$ contains both $p$ and $p'$. Hence $p$ and $p'$ have the same incidence pattern with respect to the selected three edges, impossible in $\VD_3^\circ$.
\end{proof}

\section{A fixed-uniformity trace problem for the loose triangle}\label{sec:loose}

We conclude with some auxiliary results about the $3$-uniform loose triangle $C_3^3$. It is one of the smallest configurations appearing inside $\VD_3^\circ$. We show that its unrestricted trace extremal number is cubic, but in every fixed uniformity, the order drops to quadratic.

\begin{observation}\label{obs:loose-unrestricted}
The maximum size of a set system on $n$ vertices with no trace copy of $C_3^3$ is $\Theta(n^3)$.
\end{observation}

\begin{proof}
The upper bound follows from the theorem of Keevash, Leader, Long and Wagner~\cite{KLLW}, since $C_3^3$ is  obtained as a trace of $\VD_3^\circ$. For the lower bound, take the standard product construction with two singleton factors and one chain factor. More formally, split the ground set into three parts \(X\cup Y\cup Z\), each of linear size. Order \(Z=\{z_1,\ldots,z_m\}\), and put \(Z_t=\{z_1,\ldots,z_t\}\). Let

\[\mathcal F=\{\{x,y\}\cup Z_t:x\in X,\ y\in Y,\ 0\le t\le m\}.
\]

Then \(|\mathcal F|=|X||Y|(m+1)=\Omega(n^3)\). We claim that \(\mathcal F\) contains no trace copy of \(C_3^3\). Take three members \(A_i=\{x_i,y_i\}\cup Z_{t_i}\) for $i=1,2,3$ and assume \(t_1\le t_2\le t_3\). The chain part \(Z\) contributes only to regions of the form \(\{3\},\{2,3\},\{1,2,3\}\), and hence contributes nothing to \(R_{\{1\}},R_{\{1,2\}},R_{\{1,3\}}\). In a trace copy of \(C_3^3\), the three pair-region \(R_{\{1,2\}},R_{\{1,3\}},R_{\{2,3\}}\) and the three singleton regions \(R_{\{1\}},R_{\{2\}},R_{\{3\}}\) must all be non-empty. Since \(Z\) cannot supply \(R_{\{1,2\}}\) or \(R_{\{1,3\}}\), these two regions must be supplied by the two singleton coordinates \(X\) and \(Y\). But then those two coordinates are already used to create \(R_{\{1,2\}}\) and \(R_{\{1,3\}}\), and neither can also supply \(R_{\{1\}}\). The chain part cannot supply \(R_{\{1\}}\) either. Thus \(R_{\{1\}}\) is empty, a contradiction.
\end{proof}

The proof of Theorem~\ref{thm:loose-triangle-trace} uses a pruning step and the Erd\H{o}s-Rado sunflower lemma. We say that an $r$-uniform hypergraph $\mathcal F$ is $C$-nuclear if every pair of vertices is contained in either no edge or at least $C$ edges. We recall first the sunflower lemma.

\begin{lemma}[Erdős-Rado sunflower lemma \cite{ErdosRado}]\label{lem:ER-sunflower}
For every pair of positive integers \(s\) and \(t\), there is a number
\(N(s,t)\) such that every family of \(s\)-element sets with at least
\(N(s,t)\) members contains a sunflower with \(t\) petals. That is, it
contains distinct sets \(S_1,\ldots,S_t\) and a set \(K\) such that
\(S_i\cap S_j=K\) for all \(i\neq j\).
\end{lemma}

\begin{lemma}\label{lem:nuclear}
For every pair of positive integers $D$ and $r\ge3$, there is a constant $C=C(D,r)$ with the following property. If $\mathcal F$ is a non-empty $C$-nuclear $r$-uniform hypergraph, then there are distinct vertices $x,y,z$ such that each of the three families of edges containing exactly two of $x,y,z$ has size at least $D$.
\end{lemma}

\begin{proof}
Set $C=D(r-2)+2$. Suppose the conclusion fails, and choose an edge $A=\{a_1,\ldots,a_r\}\in\mathcal F$. We claim by induction on $\ell\ge2$ that every $\ell$-subset of $A$ is contained in at least $C-D(\ell-2)$ edges. The case $\ell=2$ is exactly $C$-nuclearity.

Assume the claim for $\ell$ and consider $B=\{a_1,\ldots,a_{\ell+1}\}$. Applying the negation of the lemma to $a_1,a_2,a_3$, we may relabel so that fewer than $D$ edges contain $a_1,a_2$ and avoid $a_3$. By induction, the set $\{a_1,a_2,a_4,\ldots,a_{\ell+1}\}$ is contained in at least $C-D(\ell-2)$ edges. All but fewer than $D$ of these edges also contain $a_3$, so $B$ is contained in at least $C-D(\ell-1)$ edges. Taking $\ell=r$ says that the edge $A$ is contained in at least two distinct $r$-edges, contradiction.
\end{proof}

\begin{lemma}\label{lem:sunflower-selection}
For every $s\ge1$ there is $D_0=D_0(s)$ with the following property. If $\mathcal A,\mathcal B,\mathcal C$ are three families of $s$-sets and each has size at least $D_0$, then there are $A\in\mathcal A$, $B\in\mathcal B$ and $C\in\mathcal C$ such that each of $A,B,C$ has a private element with respect to the other two.
\end{lemma}

\begin{proof}
Apply Lemma~\ref{lem:ER-sunflower} with $t=10s+10$, and let $D_0$ be the resulting bound. Choose sunflowers $\mathcal A'\subseteq\mathcal A$, $\mathcal B'\subseteq\mathcal B$ and $\mathcal C'\subseteq\mathcal C$, each with $t$ petals and with cores $K_A,K_B,K_C$. Since petals in a sunflower are pairwise disjoint, at most $2s$ petals of $\mathcal A'$ meet $K_B\cup K_C$ outside $K_A$, and similarly for the other two sunflowers. After deleting these exceptional petals, choose one remaining petal from each sunflower so that the three chosen petals are pairwise disjoint. The corresponding sets have private elements in their petals.
\end{proof}

\begin{lemma}\label{lem:three-pairs-to-trace}
For every fixed $r\ge3$ there is $D_1=D_1(r)$ with the following property. Let $x,y,z$ be distinct vertices, and let $\mathcal A,\mathcal B,\mathcal C$ be families of $r$-sets such that every $A\in\mathcal A$ contains $x,y$ but not $z$, every $B\in\mathcal B$ contains $x,z$ but not $y$, and every $C\in\mathcal C$ contains $y,z$ but not $x$. If all three families have size at least $D_1$, then some $A\in\mathcal A$, $B\in\mathcal B$, $C\in\mathcal C$ have a trace isomorphic to $C_3^3$.
\end{lemma}

\begin{proof}
Apply Lemma~\ref{lem:sunflower-selection} with $s=r-2$ to the projected families $\{A\setminus\{x,y\}:A\in\mathcal A\}$, $\{B\setminus\{x,z\}:B\in\mathcal B\}$ and $\{C\setminus\{y,z\}:C\in\mathcal C\}$. For $D_1$ large enough, we obtain $A,B,C$ and private vertices $a\in A\setminus(B\cup C)$, $b\in B\setminus(A\cup C)$ and $c\in C\setminus(A\cup B)$. On the vertex set $\{x,y,z,a,b,c\}$, the traces of $A,B,C$ are $\{x,y,a\}$, $\{x,z,b\}$ and $\{y,z,c\}$, which form $C_3^3$.
\end{proof}

\begin{proof}[Proof of Theorem~\ref{thm:loose-triangle-trace}]
For the lower bound, fix a core $Q$ of size $r-2$ and take all edges $Q\cup X$, where $X$ ranges over all pairs from the remaining vertices. This gives $\Theta_r(n^2)$ edges. If a trace avoids $Q$, then every trace edge has size at most two; if it meets $Q$, then all trace edges have a common vertex. In either case the trace cannot be $C_3^3$.

For the upper bound, let $D_1=D_1(r)$ be given by Lemma~\ref{lem:three-pairs-to-trace}, and let $C=C(D_1,r)$ be given by Lemma~\ref{lem:nuclear}. Starting with an $r$-graph $\mathcal F$, repeatedly delete every edge containing a pair whose current positive codegree is smaller than $C$. Each pair is responsible for at most $(C-1)$ deleted edges, so fewer than $(C-1)\binom n2$ edges are deleted. Thus, if $|\mathcal F|$ is larger than a sufficiently large constant times $n^2$, a non-empty $C$-nuclear subhypergraph remains. By Lemma~\ref{lem:nuclear}, there are vertices $x,y,z$ such that the three families of edges containing $xy$ but not $z$, containing $xz$ but not $y$, and containing $yz$ but not $x$ each have size at least $D_1$. Lemma~\ref{lem:three-pairs-to-trace} then gives a trace copy of $C_3^3$.
\end{proof}

\subsection* {Concluding remarks}

The main open problem is to determine the order of $\extr(n,\VD_k)$. We expect the lower bound $\Om_k(n^{2^{k-1}-1})$ to be sharp, in agreement with the Anstee-Sali prediction. The proof of Theorem~\ref{thm:main-trace} suggests a more local problem: determine how many partial Venn regions can be forced at a given exponent. In the present argument, when passing from $k$ to $k+1$ sets, two additional non-outside regions are obtained for free. Any systematic improvement in this number would immediately improve the exponent in Theorem~\ref{thm:main-trace}.

The fixed-uniformity version also seems interesting in its own right. Even in the natural uniformity $2^{k-1}$, closing the gap between the construction in Theorem~\ref{thm:uniform-lower} and proving a stronger upper bound would require substantially more structure. 
\subsubsection*{Acknowledgements}
This project has received funding from the European Union's Horizon 2020 research and innovation programme under the Marie Sk\l{}odowska-Curie grant agreement No.~823748. It was carried out during the DIMACS REU in 2024 under the guidance of Bhargav Narayanan and Milan Haiman. We thank them for many helpful meetings and discussions.
\subsubsection*{Declaration of AI use}
The authors used OpenAI's ChatGPT for proofreading and stylistic improvements of the text. All the mathematical ideas and proofs are due to the authors.
\bibliographystyle{abbrv}
\bibliography{venn_diagrams_references}

@article{AnsteeSaliSurvey,
  author  = {Richard P. Anstee and Attila Sali},
  title   = {A survey of forbidden configuration results},
  journal = {The Electronic Journal of Combinatorics},
  year    = {2013},
  pages   = {Dynamic Survey DS20},
  note    = {Updated version, March 14, 2025},
  doi     = {10.37236/2379}
}

@article{CsakanyKahn,
  author  = {Rita Cs{\'a}k{\'a}ny and Jeff Kahn},
  title   = {A homological approach to two problems on finite sets},
  journal = {Journal of Algebraic Combinatorics},
  volume  = {9},
  number  = {2},
  pages   = {141--149},
  year    = {1999},
  doi     = {10.1023/A:1018630111976}
}

@article{ErdosTuran,
  author  = {Paul Erd{\H{o}}s},
  title   = {On extremal problems of graphs and generalized graphs},
  journal = {Israel Journal of Mathematics},
  volume  = {2},
  pages   = {183--190},
  year    = {1964}
}

@article{ErdosRado,
  author  = {Paul Erd{\H{o}}s and Richard Rado},
  title   = {Intersection theorems for systems of sets},
  journal = {Journal of the London Mathematical Society},
  volume  = {35},
  pages   = {85--90},
  year    = {1960}
}

@article{FranklFuredi,
  author  = {P{\'e}ter Frankl and Zolt{\'a}n F{\"u}redi},
  title   = {Exact solution of some {Tur{\'a}n}-type problems},
  journal = {Journal of Combinatorial Theory, Series A},
  volume  = {45},
  number  = {2},
  pages   = {226--262},
  year    = {1987},
  doi     = {10.1016/0097-3165(87)90016-1}
}

@article{FurediJiang,
  author  = {Zolt{\'a}n F{\"u}redi and Tao Jiang},
  title   = {Hypergraph {Tur{\'a}n} numbers of linear cycles},
  journal = {Journal of Combinatorial Theory, Series A},
  volume  = {123},
  pages   = {252--270},
  year    = {2014},
  doi     = {10.1016/j.jcta.2013.12.009}
}

@inproceedings{GuptaLeeLi,
  author    = {Anupam Gupta and Euiwoong Lee and Jason Li},
  title     = {The number of minimum {$k$}-cuts: Improving the {Karger--Stein} bound},
  booktitle = {Proceedings of the 51st Annual ACM SIGACT Symposium on Theory of Computing},
  series    = {STOC 2019},
  pages     = {229--240},
  publisher = {ACM},
  year      = {2019},
  doi       = {10.1145/3313276.3316395}
}

@article{KLLW,
  author  = {Peter Keevash and Imre Leader and Jason Long and Adam Zsolt Wagner},
  title   = {The extremal number of {Venn} diagrams},
  journal = {Proceedings of the American Mathematical Society},
  volume  = {154},
  pages   = {3281--3293},
  year    = {2026},
  doi     = {10.1090/proc/15402}
}

@article{PerlesShelah,
  author  = {Saharon Shelah},
  title   = {A combinatorial problem; stability and order for models and theories in infinitary languages},
  journal = {Pacific Journal of Mathematics},
  volume  = {41},
  number  = {1},
  pages   = {247--261},
  year    = {1972}
}

@article{Raggi,
  author  = {Miguel Raggi},
  title   = {Forbidden configurations: Finding the number predicted by the {Anstee--Sali} conjecture is {NP}-hard},
  journal = {Ars Mathematica Contemporanea},
  volume  = {10},
  number  = {1},
  pages   = {1--8},
  year    = {2016},
  doi     = {10.26493/1855-3974.438.300}
}

@article{Sauer,
  author  = {Norbert Sauer},
  title   = {On the density of families of sets},
  journal = {Journal of Combinatorial Theory, Series A},
  volume  = {13},
  number  = {1},
  pages   = {145--147},
  year    = {1972}
}

@article{VC,
  author  = {Vladimir N. Vapnik and Alexey Ya. Chervonenkis},
  title   = {On the uniform convergence of relative frequencies of events to their probabilities},
  journal = {Theory of Probability and Its Applications},
  volume  = {16},
  number  = {2},
  pages   = {264--280},
  year    = {1971}
}
\end{document}